\documentclass[11pt,a4paper,oneside,headings=normal, parskip=half]{scrartcl}
\usepackage[left=3cm,right=3cm,top=2cm,bottom=2.5cm]{geometry}			
\usepackage[table,xcdraw]{xcolor}
\usepackage{setspace}	   
\usepackage{aligned-overset}
\usepackage[bottom,hang]{footmisc}										
\usepackage[T1]{fontenc}												
\usepackage[utf8]{inputenc}											
\usepackage[ngerman, english]{babel}									
\usepackage{mathrsfs}
\usepackage{amsmath}													
\allowdisplaybreaks
\usepackage{amsthm}														
\usepackage{amstext}	
\usepackage{abstract}
\usepackage{amssymb}													
\usepackage{amsfonts}													
\usepackage{dsfont} 													
\usepackage{mathtools}
\usepackage{url}														
\usepackage{graphicx}													
\usepackage{tocloft}													
\usepackage[bf,hang,nooneline,justification=centering]{caption}			
  
\usepackage{tabularx}													
\usepackage{multirow}													
\usepackage{booktabs}													
\usepackage{acronym}
\usepackage[titletoc,title]{appendix}									
\usepackage[
backend=biber,
style=alphabetic,
]{biblatex}
\usepackage{csquotes}
\usepackage{placeins}
\usepackage{lmodern} 													
\usepackage{xcolor}														
\usepackage[hidelinks]{hyperref}													
\usepackage{stmaryrd}
\usepackage{comment}

\def\XXint#1#2#3{{\setbox0=\hbox{$#1{#2#3}{\int}$ }
\vcenter{\hbox{$#2#3$ }}\kern-.6\wd0}}

\usepackage{paralist}
\usepackage{enumitem}
\usepackage{orcidlink}

\newcounter{assumption}
\theoremstyle{definition} 
\newtheorem{theorem}{Theorem}[section]                      
\newtheorem{corollary}[theorem]{Corollary} 
\newtheorem{example}[theorem]{Example}

\newtheorem{lemma}[theorem]{Lemma}
\newtheorem{definition}[theorem]{Definition}

\newenvironment{beweis}{\begin{proof}[Proof]}{\end{proof}}

\newtheorem{bemerkung}[theorem]{Remark}

\makeatletter
\newcommand{\AlignFootnote}[1]{%
  \ifmeasuring@
  \else
    \iffirstchoice@
      \footnote{#1}%
    \fi
  \fi}
\makeatother

\DeclareMathOperator*{\fast}{-a.s.}

\newcommand{\dd}{\mathrm{d}}

\newcommand{\E}{\mathbb{E}}

\newcommand{\fastsicher}{\quad \W \fast}

\newcommand{\W}{\mathbb{P}}

\newcommand{\skalarq}[1]{ \langle #1 \rangle }

\DeclareMathOperator{\ccc}{C}

\newcommand*{\C}{\mathbb{C}}

\newcommand{\N}{\mathbb{N}}

\newcommand{\R}{\mathbb{R}}

\newcommand{\Rd}{\R^{d}}

\renewcommand{\phi}{\varphi}
\renewcommand{\epsilon}{\varepsilon}
\newcommand{\eps}{\varepsilon}

\DeclareMathOperator{\llll}{L}

\newcommand{\PP}{\mathcal{P}}

\newcommand{\abs}[1]{\left\vert #1 \right\vert}

\newcommand{\lp}[1]{\llll^{#1}}

\newcommand{\Pm}{\PP_2(M)}

\graphicspath{ {Bilder/} }
\usepackage{tikz}
\usetikzlibrary{arrows,automata,shapes,calc}
\usepackage{setspace}
\begin{document}

   \begin{center}
    \Large
    \textbf{Krylov-Veretennikov decomposition for measure-valued processes induced by SDEs with interaction on Riemannian manifolds}
        
        
    \vspace{0.4cm}
    Andrey Dorogovtsev\footnote{NAS Ukraine, andrey.dorogovtsev@gmail.com}\orcidlink{0000-0003-0385-7897}, Alexander Weiß\footnote{University of Leipzig, alexander.weiss@math.uni-leipzig.de}\orcidlink{0009-0006-2791-4912}

    \vspace{0.9cm}
\end{center}

 \begin{abstract}
     We introduce a framework for stochastic differential equations (SDEs) with interaction on compact, connected, $d$-dimensional manifolds. For SDEs whose drift and diffusion coefficients may depend on both the state variable and the empirical distribution, we establish existence and uniqueness of strong solutions under general regularity assumptions.
We study the associated measure-valued process on the Wasserstein space over the manifold, deriving an explicit Itô–Wiener decomposition. We prove Malliavin differentiability of the solution and, using directional derivatives in the Wasserstein space, establish smooth dependence of the solution on the measure component for a class of coefficients.
 \end{abstract}

\section{Introduction}
In this paper, we introduce the notion of stochastic differential equations with interactions on compact, connected, $d$-dimensional Riemannian manifolds $M$, as introduced by Dorogovtsev \cite{dorogovtsev1997smooth}.
These equations are of the following type in the case $M=\Rd$  
\begin{align*}
    \begin{cases}
    \dd x(u,t)&= a(x(u,t),\mu_t)\dd t + \int_{\Rd} b(x(u,t),\mu_t,p) W(\dd p, \dd t)\\
    x(u,0)&= u \in \Rd \\
    \mu_t&= \mu \circ x^{-1}(\cdot,t).
    \end{cases} 
\end{align*}
Here, $\mu$ is a probability measure on $\Rd$, and $W$ is a Wiener sheet on $\lp{2}(\Rd)$.We can interpret points in $\Rd$ as particles distributed in space according to $\mu$, and the flow $x$ as the temporal evolution of these particles. This evolution depends not only on individual particle dynamics but also on the evolution of their spatial distribution, reflecting particle interaction. The fact that equations with interaction are stochastic flows is advantageous for generalizing them to manifolds and simplifies their treatment. There has been done a lot of work in the realm of SDEs with interaction on $\Rd$, such as asymptotics and intermittency phenomena \cite{belozerova2020asymptotic,DorWeiss}, reflected SDEs with interaction \cite{chen2024exponential} or existence and uniqueness of the SPDE induced from the measure-valued process and its properties \cite{gess2022conservativespdesfluctuatingmean,gess2024stochastic}.

To our knowledge, the first and only consideration of SDEs with interactions on manifolds was conducted by Ding, Fang, and Li \cite{ding2023stochasticdifferentialequationsstochastic}, where the authors considered special interaction kernels in the drift but no interaction term in the noise.
We are interested in ergodicity results for the measure-valued process $(\mu_t)_{t \ge 0}$. While the Krylov–Bogoliubov theorem ensures the existence of an invariant measure on the space $\PP_2(M)$, ergodicity results cannot be established using classical techniques. This is due to the fact that the Markov process $(\mu_t)_{t \ge 0}$ takes values in an infinite-dimensional and non-linear space, while the driving noise is finite-dimensional. Hence, "smoothing" properties of the induced semigroup $P_t f(\mu) = \E(f(\mu_t))$ cannot be expected, due to the highly degenerate noise.

Therefore, we study the Krylov–Veretennikov decomposition of functionals $(F(\mu_t))_{t \ge 0}$, where $F: \PP_2(M) \to \R$. This provides a formula for the kernels of the Itô–Wiener decomposition, offering insight into the structure of functionals of $(\mu_t)_{t \ge 0}$.The explicit formula enables manual calculations of long-term behavior by analising individual terms of the decomposition separately. To our knowledge this is the first contribution towards a Krylov-Veretennikov decomposition for the measure valued process induced by SDEs with interaction. 

The Krylov-Veretennikov decomposition has been established in \cite{veretennikov1976explicit} by Krylov and Veretennikov. In their work they considered solutions to ordinary one-dimensional SDEs and described the Itô-Wiener kernels only through the semigroup induced by the SDE and a differential operator, they showed with the help of the decomposition existence of strong solutions, the same ansatz has been applied to SDEs with more general coefficients  \cite{krylov2000direct,krylov2025strong}. After the work of Krylov and Veretennikov, there have been more general considerations, in \cite{dorogovtsev2012krylov} Dorogovtsev showed a Krylov-Veretennikov type decomposition for general stochastic semigroups on Hilbert spaces. Moreover he showed the Krylov-Veretennikov decomposition for the Arratia flow. In \cite{riabov2015krylov} the Krylov-Veretennikov decomposition for a stopped Brownian motion has been proven and in \cite{glinyanaya2015krylov} the Krylov-Veretennikov for functionals of the Harris flow. To our knowledge this is the first paper treating the Krylov-Veretennikov decomposition result for the measure valued process induced by SDEs with interactions. 

The structure of the paper is as follows. In Chapter 2, we present a definition and a uniqueness and existence result for SDEs with interaction on general Riemannian manifolds. In Chapter $3$ we present differential operators on the Wasserstein space and discuss smoothness of the solution to SDEs with interaction with respect to these differential operators. 
In chapter 4 we present and prove our main result which in simplified form looks as follows

    \begin{theorem}
        Let $f$ be a "smooth" and bounded functional then 
        \begin{align*}
            f(\mu_t) =    T_tf(\mu) +\sum^n_{i=1}\sum^\infty_{k=1} \underset{\Delta^k([0,t])}{\int \dots \int} T_{\tau_1}A_i T_{\tau_2-\tau_1}\dots T_{\tau_k-\tau_{k-1}}A_iT_{t-\tau_k}f(\mu) \dd B^i(\tau_1)\dots \dd B^i(\tau_k)    \end{align*}
            where $A_i$ is a differential operator on the Wasserstein space more precisely a directional derivative into the direction of the noise vector fields and $T_tf(\mu)=\E(f(\mu_t))$.
    \end{theorem}

   The form of our Krylov-Veretennikov expansion is very close to the original in  \cite{veretennikov1976explicit}the authors obtained for $T_tf(u)= \E(f(y(u,t)))$ and $y(u,t)$ solves a one dimensional SDE
   \begin{align*}
       \begin{cases}
           dy(u,t)&=\alpha(y(u,t))\dd t + \beta(y(u,t))\dd B_t\\
           x(u,0)&= u\in \R
       \end{cases} 
   \end{align*}
where $\alpha,\beta:\R\to \R$. Their result was
   \begin{align*}
       f(x(u,t))= E(f(x(u,t))) +\sum^\infty_{k=0} \underset{\Delta([0,t]^k)}{\int \dots \int} T_{t-\tau_k} \beta(u)\frac{\partial}{\partial u} T_{\tau_k-\tau_{k-1}} \dots \beta(u)\frac{\partial}{\partial u}T_{\tau_1} f(u)\dd B(\tau_1) \dots \dd B_{\tau_k}.
   \end{align*}
One can observe that the decomposition here only depends on the semigroup and directional derivative operator $b\frac{\partial}{\partial u}$ where $b$ is the noise coefficient of the SDE describing $x(u,t)$ similar to our main result.

\section{Existence and Uniqueness of SDEs with interaction on compact Riemannian manifolds}
Let $M$ be a compact smooth connected $d$-dimensional Riemannian manifold without boundary or $M=\Rd$. Let $\PP_2(M)$ be the space of all probability measures with finite second moment, note that in the compact case $\PP_2(M)$ is the space of all probability measures. Furthermore define the Wasserstein-$2$-distance $\gamma_2$ by 
\begin{align*}
    \gamma_2(\mu,\nu):= \left(\inf_{\kappa \in \C(\mu,\nu)} \underset{M\times M}{\int\int} d^2_M(u,v) \kappa(\dd u,\dd v)\right)^{\frac{1}{2}},
\end{align*}
where $C(\mu,\nu):= \{\rho \in M\times M : \rho(\cdot\times M)= \mu; \rho(M\times \cdot)=\nu\}$ is the space of all coupling over $\mu$ and $\nu$ and $d_M$ is the geodesic distance on $M$.
Let $(V_0,\dots,V_n)$ be mappings such that $V_i:M\times \PP_2(M)\to TM$ for all $i=0,\dots,n$, such that $V_i(\cdot,\mu)$ is a smooth vector field for all $\mu\in \PP_2(M)$ and $V_i(x,\cdot):\PP_2(M)\to T_xM$ and its derivative are Lipschitz map with respect to the Wasserstein-$2$ distance $\gamma_2$ on $\PP_2(M)$ for all $x\in M$, moreover the partial derivatives of $V_i$ suffice the same properties.

We consider equations of the following form 
\begin{align}\label{Wechselwirkung}
    \begin{cases}
        dx_\mu(u,t)&= F(x_\mu(u,t),\circ \dd t)\\
        x_\mu(u,0)&= u\in M \\
        F(u,t)&= \int_0^t V_0(u,\mu_t) \dd t + \sum^n_{i=1}\int_0^t V_i(u,\mu_t) \dd B^i_t \\
        \mu_t &= \mu \circ x_\mu(\cdot,t)^{-1}
    \end{cases}
\end{align}
An $M$-valued process $x$ solves the system \eqref{Wechselwirkung} if for all $f\in \ccc^\infty(M)$ 
\begin{align*}
    f(x_\mu(u,t))&= f(u)+\int_0^t F(x(u,t),\circ \dd t)f \\
    &=f(u)+ \int \Tilde{V}_0(\cdot,\mu_t)f(x_\mu(u,s)) \dd s + \sum^n_{i=1}\int_0^t V_i(\cdot,\mu_s)f(x_\mu(u,s)) \dd B_s  
\end{align*}
where 
\begin{align*}
    \Tilde{V}_0(\cdot,\mu)f= V_0(\cdot,\mu)f +\frac{1}{2} \sum_{i=1}^n V^2_i(\cdot,\mu)f 
\end{align*}
in local coordinates $(x^1,\dots,x^d)$.

\begin{lemma}\label{Bilipschitzlemma}
    Let $M$ be a connected $d$-dimensional smooth Riemannian manifold without boundary and let $\iota: M \to \R^N$ be the Nash embedding then we have for all $x,y\in M$ and $C>0$
    \begin{align}\label{Einbettungbilip}
        \abs{\iota(x)-\iota(y)} \le d_M(x,y) \le C \abs{\iota(x)-\iota(y)}
    \end{align}

\end{lemma}
\begin{proof}
    It suffices, by isometry to show 
    \begin{align*}
        \abs{\iota(x)-\iota(y)}\le d_{\iota(M)}(\iota(x),\iota(y))\le C\abs{\iota(x)-\iota(y)}
    \end{align*}
    since the first inequality is trivial we shall only show the second. Let $x\in M$ Choose a submanifold  chart $(U_x,\phi)$ where $U_x$ is an open $\R^N$ neighbourhood around $\iota(x)$ such that $\phi(U)$ is a ball and $\phi_i^{d+1}(\iota(y))=\dots=\phi_i^N(\iota(y))=0$ for $y\in U_x\cap \iota(M)$, moreover by shrinking $U_x$ if necessary we can assume w.l.o.g that $\phi\in \ccc^\infty_b(U_x)$. Now let $ \iota(y),\iota(z) \in U_x\cap M$ .Let $\gamma(t)= \phi^{-1}_i(\phi(\iota(z))+t(\phi(\iota(y))-\phi(\iota(z))))$ which can be chosen since $\phi(U)$ is a ball. Now $\gamma(0)=\iota(x)$ and $\gamma(1)=\iota(y)$ we get $\gamma^\prime \equiv \phi(\iota(x))-\phi(\iota(y))$ in local coordinates.  
    
    We denote by $U^\eps_x$ a $\R^N$-open neighbourhood of $\iota(x)$ such that $U^\eps_x\cap M= B_{\iota(M)}(\iota(x),\eps)$ for all $\eps>0$.  Choose $\delta>0$ such that $\overline{B_{\R^N}(\iota(x),\delta)}\subset U_x$. Now let $\eps>0$ be small enough such that $B_{\iota(M)}(\iota(x),\eps)\subset B_{\R^N}(\iota(x),\delta)\cap \iota(M)$ we can choose now $U^{\frac{\eps}{2}}_x\subset U^\eps_x$  such that $U^{\eps}_x \subset  B_{\R^N}(\iota(x),\delta)$ then we have for all $y,z\in B_M(x,\eps)$
    \begin{align*}
        d_{\iota(M)}(\iota(y),\iota(z)) \le C \abs{\phi(\iota(z))-\phi(\iota(y))}. 
    \end{align*}
    where $C>0$ depends solely depends on the Riemannian metric and the chart $\phi$. Moreover 
    \begin{align*}
        d_{\iota(M)}(\iota(z),\iota(y))\le \abs{\phi(\iota(y))-\phi(\iota(z)) }\le \sup_{x\in \overline{B_{\R^N}(\iota(x),\delta)}}  \abs{\nabla \phi(x)} \abs{\iota(y)-\iota(z)}. 
    \end{align*} 
 Now cover $M$ with such $B_M(x_i,\frac{\eps_i}{2})_{i=1,\dots,m}$ as constructed above, then $B_M(x_i,\eps_i)$ is also a cover of $M$ on which we have the same coordinate chart $(\phi_i)_{i=1,\dots, n}$ as constructed. Then we can choose by construction a universal constant $C>0$ such that, whenever $y,z\in B_M(x_i,\eps_i)$ for some $i=1,\dots,n$ holds, we have 
 \begin{align}\label{AbschGeodist}
     d_{\iota(M)}(\iota(y),\iota(z))\le C \abs{\iota(y)-\iota(z)}_{\R^N}.
 \end{align}
Set $\eps = \min_{i=1,\dots,n} \eps_i $, hence for all $y,z \in M$ such that $d_M(y,z)<\frac{\eps}{2}$, we get \eqref{AbschGeodist}. To observe this, note that there exists $j\in\{1,\dots,n\}$ such that $y\in B_M(x_j, \frac{\eps_j}{2})$, hence
\begin{align*}
    d(z,x_j)\le d(y,x_j) + d(z,y)\le \frac{\eps_j}{2} +\frac{\eps}{2}\le \eps_j.
\end{align*}
This implies $y,z \in B_M(x_j,\eps_j)$. 
Assume now, that \eqref{Einbettungbilip} does not hold for $y,z\in M$ such that $d_M(y,z)\ge \frac{\eps}{2}$. Then we have sequences $y_k,z_k \in M$ for $k\in \N$ such that 
\begin{align*}
    d_M(y_k,z_k) \ge k\abs{\iota(y_k)-\iota(z_k)} 
\end{align*}
 for all $k\in \N$. By compactness we can assume w.l.o.g that $y_k \to y $ and $z_k \to z $ where $y,z\in M$ such that $y\neq z$ since $d_M(y_K,z_k)\ge \frac{\eps}{2}$. Since $\iota:M\to \R^N$ is continuos we get a contradiction. Since 
 \begin{align*}
     d_M(y,z)= \infty 
 \end{align*}
 is impossible.
 Therefore \eqref{Einbettungbilip} holds for all $y,z\in M$.  
 \end{proof}
\begin{theorem}\label{wohlgestellt}
    The solution to $\eqref{Wechselwirkung}$ exists and is unique.     
\end{theorem}
\begin{beweis}

Let $ \mu\in \Pm$
 consider the following equation  
\begin{align}\label{ohneWechselwirkung}
    \begin{cases}
        dx^0_\mu(u,t)&= \int_0^t F_0(x^0_\mu(u,t),\circ \dd t)\\
        x^0_\mu(u,0)&= u \in M\\
    \end{cases}
\end{align}
where 
\begin{align*}
F_0(t)= V_0(\cdot,\mu)t + \sum_{i=1}^n V_i(\cdot,\mu)B^i_t    
\end{align*}
by \cite{kunita1990stochastic} there exists a unique strong solution to the equation \eqref{ohneWechselwirkung}. Note that this solution is a stochastic flow  of diffeomorphisms and thus 
\begin{align*}
    \lim_{t\to s}\gamma^2_2\left(\mu\circ (x_\mu^0(\cdot,t))^{-1},\mu\circ (x^0_\mu(\cdot,s))^{-1}\right)\le \lim_{t\to s}\int_M \gamma_2(x^0_\mu(u,t),x^0_\mu(u,s))^2\mu(\dd u)= 0 
\end{align*}
where the last equality follows from continuity for all $u$ almost surely, as well as from compactness of $M$ which allows us to use the dominated convergence theorem. 
Consider now the iteratively for $n\in \N$
\begin{align*}
    F_n(t)= \int_0^t V_0(\cdot,\mu^{n-1}_s) \dd s + \sum^n_{i=1}\int_0^t V_i(\cdot,\mu^{n-1}_s) \dd B_s 
\end{align*} 
where $\mu_t^{n-1}= \mu\circ x^{n-1}_\mu(\cdot,t)$. Note that 
this is is a vectorfield valued semimartingale. Again by \cite{kunita1990stochastic} there exists a unique strong solution to the equation 
\begin{align*}
    \begin{cases}
        dx^n_\mu(u,t)&= \int_0^t F_{n-1}(x^n_\mu(u,t),\circ \dd t)\\
        x^n_\mu(u,0)&= u \in M.
    \end{cases}
\end{align*}
Now we have to show that $ \mu^n$ and $x_\mu^n$ converge in the space $\lp{2}(\Omega,\ccc([0,T],\Pm))$ and $\lp{2}(\Omega,\ccc([0,T],M))$ respectively, where  
\begin{align*}
    \lp{2}(\Omega,S):=\{f:\Omega \to S : \E(d_S(o,f)^2)<\infty, o\in S \}
\end{align*}
for arbitrary polish spaces $S$. 
 To see this consider the Nash imbedding of $M$ in some $\R^N$ as a closed submanifold. Since the Nash imbedding is a diffeomorphism between compact Riemannian manifolds, it is bilipschitz and due to the isometry property of the imbedding the vector field remain Lipschitz with respect to the Wasserstein distance. By enlarging the vector fields to the entire $\R^N$ in the following way by the Nash imbedding we have vector fields $\Bar{V_i}(\cdot,\mu):= d\iota(V_i(\cdot,\mu))$ note moreover 
 \begin{align*}
     d \iota_x: T_x(M) \hookrightarrow T_x(\R^N) \cong \R^N
 \end{align*} 
now since $\iota(M)$ is a compact submanifold there exists $m\in \N$ and $(x_i,U_{i})_{i=1,\dots m}$ with $x_i\in M$ and $U_i$ an $\R^N$-open neighbourhood $U_i$ as well as  smooth functions $\phi_i:U_i\to \R^N$ with the property $\phi^{d+1}_i(x)=\dots=\phi_i^{N}(x)=0$ for all $x\in M$ and $\phi(U_i)$ is a Ball. We then define for $(u_1,\dots,u_N)=u\in  B_{\R^N}(\eps,x_i)$ 
\begin{align*}
    \Bar{V}^j_i(u,\mu):= \Bar{V}_i(\phi_j^{-1}(u_1,\dots,u_d,0,\dots,0),\mu)  
\end{align*}
for all $\mu$ and all $i=1,\dots,n, j=1,\dots,m$. Let $(\rho_j)_{j=1,\dots,m+1}$  be a smooth partition of unity subordinate to $(\phi_j(U_j))_{j=1,\dots,m+1}$ where 
\begin{align*}
    \phi_{m+1}(U_{m+1}):= \R^N\setminus(\cup^{m+1}_{j=1}\phi_j(U_j))
\end{align*}
then 
\begin{align*}
    \Bar{V}_i(u,\mu)= \sum^{m+1}_{j=1} \rho_j \Bar{V}^j_i(u,\mu) 
\end{align*}
 where $ \Bar{V}^j_i\equiv 0$ for all $i=1,\dots, n$. A similar construction can be done for the vector field $\Tilde{V}_0$. By definition $\Bar{V}_i(\cdot,\mu)$ is smooth for all $\mu\in \PP_2(M)$ and $\Bar{V}_i(\Bar{u},\cdot)$ is Lipschitz for all $\Bar{u}\in \R^N$ for $i=1,\cdot,n$ the same holds for $\Bar{\Tilde{V}}_0$   
 we get that $\iota(x^n):= \Bar{x}^n$ solves the following equation
\begin{align*}
    \begin{cases}
        d\Bar{x}^n(\Bar{u},t)&= \Bar{\Tilde{V_0}}(\Bar{x}^n,\mu^{n-1}_t) \dd t + \sum^n_{i=1} \Bar{V}_i(\Bar{x}^n(u,t),\mu^{n-1}_t) \dd B^i_t. \\
       \Bar{x}^n(\Bar{u},0) &= \Bar{u} \quad \forall u\in M
    \end{cases}
\end{align*}
where $\Bar{V_i}$ are the vector fields $V_i$ enlarged to $\R^N$ for all $i=0,\dots,n$. It is important to note that $\Bar{x}$ stays in $\iota(M)$ for all $t\ge 0$ a.s. for all $n\in \N$.  
It is important to note that 
\begin{align*}
    \mu^{n-1}_t =\mu \circ \left(\iota^{-1}(\Bar{x}^n (\iota(\cdot),t))\right)^{-1} 
\end{align*}
Then by standard procedure, see  \cite{Dorogovtsev+2024}   we can thus derive for all $\Bar{u}\in \iota(M)$.

\begin{align*}
    \E(\sup_{0\le t \le T}\abs{\Bar{x}^{n+1}(\Bar{u},t)-\Bar{x}^{n}(\Bar{u},t)}^p) \le C \int_0^T \E(\gamma^M_2(\mu^{n}_t,\mu_t^{n-1})^p) \dd t. 
\end{align*}
Furthermore we get by Lemma \ref{Bilipschitzlemma}
\begin{align*}
    \E(\sup_{0\le t\le T}\gamma^M_2(\mu^n_t,\mu^{n-1}_t)^p)&\le \left(\int_M \E\left(\sup_{0\le t\le T}d_M\left(x^n(u,t),x^{n-1}(u,t)\right)^2\right) \mu(\dd u)\right)^{\frac{p}{2}} \\
    &\le C \int_0^T \E(\gamma^M_2 (\mu^{n-1}_t,\mu_t^{n-2})^p) \dd t 
\end{align*}
therefore the convergence of $\mu^n_t$ and $\Bar{x}^n(u,t)$ can be checked as usual. Hence $\Bar{x}^n$ and $\mu^n_t$ converge to some random fields $\Bar{x}_t$ and $\mu_t$ thus 
\begin{align*}
    F(t)f= \int_0^t V_0  (\cdot,\mu_s)f \dd s+ \sum^n_{i=1} \int_0^t V_i(\cdot,\mu_s)f \dd B^i_s 
 \end{align*}
and $x_\mu(u,t)=\iota^{-1}(\Bar{x}(\iota(u),t))$ solves \eqref{Wechselwirkung}. Note moreover that up to a subsequence, we have to show that $\mu_t= \mu \circ x_\mu^{-1}$. 
\begin{align*}
 \E(\gamma^2_2(\mu_t,\mu\circ x_\mu^{-1}(\cdot,t)))&= \lim_{n\to \infty}\E\left( \gamma_2^2(\mu^n_t,\mu_t)\right)\\
 &\le\lim_{n\to \infty}\int_M\E( d^2_M(\iota(x^n_\mu(\iota(u),t),x_\mu(u,t))) \mu(\dd u)=0
\end{align*}
implying $\mu_t=\mu\circ x_\mu^{-1}(\cdot,t)$. 
Uniqueness of the solutions follows from the uniqueness in $\R^N$ see \cite{Dorogovtsev+2024}.  
 \end{beweis}
\begin{lemma}\label{Massabsch}
Let $ \mu,\nu \in \PP_2(M) $ then 
\begin{align*}
    \E(\sup_{0\le t\le T}\gamma^p_2(\mu_t,\nu_t))\le C \gamma^p_p(\mu,\nu)
\end{align*}
for all $p\ge 2$ where 
\begin{align*}
    \gamma^p_p(\mu,\nu) = \inf_{\kappa\in C(\mu,\nu)} \underset{M\times M}{\int \int} d^p_M(u,v)\kappa(\dd u,\dd v)
\end{align*}
\end{lemma}
\begin{proof}
    Let $\iota : M \to \iota(M)\subset \R^N$ for some $N\in \N$ be the Nash embedding. And consider $\Bar{x}_\mu(\iota(u),t):=\iota(x_\mu(u,t))$ and 
    \begin{align*}
        \begin{cases}
            d\Bar{x}_\mu(v,t)&= \Bar{\Tilde{V}}_0(\Bar{x}_\mu(v,t),\mu_t)\dd t + \sum_{i=1}^n\Bar{V}_i (\Bar{x}_\mu(v,t),\mu_t)\dd B_t^i  \\
        \Bar{x}_\mu(v,0)&= v \qquad \forall v\in \R^N
        \end{cases} 
    \end{align*}
    where as usual the vector fields are enlarged on the entire $\R^N$
    Then one can show 
    \begin{align*}
        \E(\sup_{0\le t\le T} \abs{\Bar{x}_\mu(v_1,t)-\Bar{x}_\mu(v_2,t)}^p )\le C\abs{v_1-v_2}^p.
    \end{align*}
moreover consider for $\mu,\nu\in \PP_2(M)$ and $p\ge 2$
\begin{align*}
   & \E(\sup_{0\le t\le T} \abs{\Bar{x}_\mu(u,t)-\Bar{x}_\nu(u,t)}^p)\\
   \le& \int_0^T\E(\abs{\Bar{x}_\mu(u,t)-\Bar{x}_\nu(u,t)}^p+\gamma^p_2(\mu_t,\nu_t))\dd t 
\end{align*}
Grönwall's inequality thus imply 
\begin{align*}
   \E(\sup_{0\le t\le T} \abs{\Bar{x}_\mu(u,t)-\Bar{x}_\nu(u,t)}^p)\le C \E\left(\sup_{0\le t\le T}\gamma^p_2(\mu_t,\nu_t)\right)
\end{align*}
 Now consider for the optimal coupling $\kappa\in C(\mu,\nu)$
 \begin{align*}
     \E(\gamma_2(\mu_t,\nu_t)^p)&\le \left(\underset{M\times M}{\int\int} \E(d^2_M(x_\mu(u,t),x_\nu(v,t))) \kappa(\dd u,\dd v)\right)^{\frac{p}{2}} \\
     &\le C \underset{M\times M}{\int\int} \E(\abs{\iota(x_\mu(\iota(u),t))-\iota(x_\nu(\iota(v),t))}^p) \kappa(\dd u,\dd v)  \\
     &\le C\Bigg(\Big(\underset{M}{\int} \E(\abs{\iota(x_\mu(\iota(u),t))-\iota(x_\nu(\iota(u),t))}^2) \mu(\dd u)\Big)^{\frac{p}{2}}  \\
     &+ \underset{M\times M}{\int\int} \E\left(\abs{\iota(x_\nu(u,t))-\iota(x_\nu(v,t))}^p\right) \kappa(\dd u,\dd v)\Bigg) \\
     &\le C(\int_0^t \gamma_2(\mu_s,\nu_s)^p \dd s + \gamma^p_p(\mu,\nu)
 \end{align*}
 Grönwalls' Lemma yields the result.
 \end{proof}
 \begin{corollary}\label{Lipschitz}
     Let $u,v\in M$, $\mu,\nu \in\PP_2(M)$ and $p\ge 2$ then 
     \begin{align*}
         \E\left(d(x_\mu(u,t),x_\nu(v,t)^p\right)\le C(d^p(u,v) +\gamma^p_p(\mu,\nu)) 
     \end{align*}
 \end{corollary}
 \begin{beweis}
    The proof follows by imbedding the equation into $\R^N$ via the Nash imbedding and enlarging the vector fields. Then we get as in Lemma \ref{Massabsch} and as in the proof of Theorem \ref{wohlgestellt}.
    \begin{align*}
        \E(d(x_\mu(u,t),x_\nu(v,t))^p) &\le C\E(\abs{\iota(x_\mu(u,t))-\iota(x_\nu(v,t))}^p)\\
        &\le C\left(\E(\abs{\iota(x_\mu(u,t))-\iota(x_\nu(u,t))}^p) +  \E(\abs{\iota(x_\nu(u,t))-\iota(x_\nu(v,t))}^p) \right)\\
        &\le C(d_M^p(u,v)+\gamma_2^p(\mu,\nu))
    \end{align*}
 \end{beweis}

\section{Smoothness of solution with respect to the measure variable}
Our goal is to prove a Krylov-Veretennikov type decomposition for the measure valued process $(\mu_t)_{t\ge 0}$. It is natural to assume that this will involve a differential operator on the Wasserstein space. Since we cannot expect regularisation to happen for the measure valued process, we have to prove smoothness of the solution with respect to the measure variable. Since our computations are mainly based on embedding the equation into a Euclidean space and enlarging vector fields, we will need a notion of measure derivative which translates naturally from a Riemannian manifold setting into the Euclidean space.

The following definition is due to \cite{ren2021derivative} for more information about measure derivatives of this type consider also the references therein.
\begin{definition}[\cite{ren2021derivative}]\label{Intrdiff}
    Let $F: \mathcal{P}_2(M)\to \R$  and let $V\in \Gamma(TM) $. Furthermore let $\Phi_\eps^V$ denotes the flow with respect to $V$ at time $\eps>0$. Then $F$ is called intrinsically differentiable in direction $V$, if  
    \begin{align*}
        \lim_{\eps\searrow 0} \frac{F(\mu\circ (\Phi^V_\eps(\cdot))^{-1})-F(\mu)}{\eps}=: D^V_IF(\mu).
    \end{align*}
Furthremore it is called intrinsically differentiable, if it is intrinsically differentiable in every direction $V\in \Gamma(TM)$ and $D^V_IF(\mu)$ is continuous with respect to the $\lp{2}(\mu,M,TM)$ norm and linear with respect to $V$.     
\end{definition}

From the definition of intrinsically differentiable it follows immediately, that $D^{\cdot}_IF(\mu): \Gamma(TM) \to \R $ can be uniquely extended continuously to $\lp{2}(\mu,M,TM)$ furthermore, for all $V\in \lp{2}(\mu,M,TM)$, there exists a function $D_IF(\mu)(\cdot): M \to TM$ such that
\begin{align*}
    D^V_I F(\mu) = \int_M \skalarq{D_IF(\mu)(u),V(u)}_{T_uM} \mu(\dd u).
\end{align*}
\begin{bemerkung}
Note that this definition can be introduced for functionals $F:\PP_2(\Rd) \to \R^d$ in exactly the same way.
\end{bemerkung}

Let us furthermore define the same notion for different types of maps from the Wasserstein space. The following definitions are straightforward adaptations from Definition \ref{Intrdiff}.
\begin{definition}
     Let $F: \mathcal{P}_2(M)\to M$  and let $V\in \Gamma(TM) $. $F$ is called intrinsically differentiable if 
    \begin{align*}
        \frac{d}{d\eps}_{\vert_{\eps=0}} F(\mu^\eps):= D^V_IF(\mu)\in T_{F(\mu)}M.
    \end{align*}
Furthremore it is called intrinsically differentiable, if it is intrinsically differentiable for all $V\in \Gamma(TM)$ and $D^V_IF(\mu)$ is continuous with respect to the $\lp{2}(\mu,M,TM)$ norm and linear with respect to $V$. 
\end{definition}
As before $D^V_IF(\mu)$ can be extended to all $V\in \lp{2}(\mu,M,TM)$ uniquely. Furthermore there exists a linear map $D^V_IF(\mu)(u): T_uM \to T_{F(\mu)}M$ such that
\begin{align*}
    D^V_IF(\mu)= \int_M D^V_IF(\mu)(u)V(u) \mu(\dd u)
\end{align*}
this can be easily shown in the following way. 
Fix a basis $(v_1,\dots,v_d)$ of $T_{F(\mu)}M$, then we get for all $V\in \lp{2}(\mu,M,TM)$ 
\begin{align*}
    D^V_IF(\mu)= \sum^d_{i=1}\skalarq{D^V_IF(\mu),v_i}_{T_{F(\mu)}M}v_i
\end{align*}
note that $V\mapsto \skalarq{D^V_IF(\mu),v_i}_{T_{F(\mu)}M} $ is linear and continuous and linear with respect to $V$. Thus there exists a function $D_I^{V,i}F(\mu)(\cdot): M \to TM$ such that 
\begin{align*}
    \skalarq{D^V_IF(\mu),v_i}_{T_{F(\mu)}M}= \int_M \skalarq{D_I^{V,i}F(\mu)(u),V(u)}_{T_uM} \mu(\dd u)
\end{align*}
for all $i=1,\dots,d$. Hence 
\begin{align*}
    D^V_IF(\mu)= \sum^d_{i=1}\int_M \skalarq{D_I^{V,i}F(\mu)(u),V(u)}_{T_uM} \mu(\dd u)v_i= \int_M \sum^d_{i=1} \skalarq{D_I^{i}F(\mu)(u),V(u)}_{T_uM}v_i \mu(\dd u)
\end{align*}
and therefore the linear map is defined as follows for all $L\in T_uM$:
\begin{align*}
    D_I^VF(\mu)(u)L := \sum^d_{i=1} \skalarq{D_I^{V,i}F(\mu)(u),L}_{T_uM}v_i
\end{align*}
which is linear and therefore continuous for all $u\in M$.
 Lastly we define the intrinsic derivative for vector field valued maps. 
 \begin{definition}
       Let $F: M\times\mathcal{P}_2(M)\to TM$  and let $V\in \Gamma(TM) $. where $F(\cdot,\mu)\in \Gamma(TM)$ for all $\mu\in \PP_2(M)$. Then $F$ is called intrinsically differentiable in direction $V$, if  
     \begin{align*}
         \frac{d}{d\eps}_{\vert_{\eps=0}} F(u,\mu^\eps):= D^V_IF(u,\mu)\in T_{u}M.
     \end{align*}
     for all $u\in M$.
 Furthermore it is called intrinsically differentiable, if it is intrinsically differentiable for all $V\in \Gamma(TM)$ and $D^V_IF(u,\mu)$ is continuous with respect to the $\lp{2}(\mu,M,TM)$ norm and linear with respect to $V$ for all $u\in M$. 
\end{definition}
 There exists a function $D_IF(u,\mu) (\cdot): TM \to T_uM$ such that
 \begin{align*}
     D_I^VF(u,\mu)= \int_M D_IF(u,\mu) (x) V(x) \mu(\dd x)
 \end{align*}
 where $D_IF(u,\mu)(x): T_xM\to T_uM$ is linear. The Example 3.5 $i)$ is due to Remark 1.1 \cite{ren2021derivative}. 
 \begin{example}
     \begin{enumerate}[label=$\roman*)$]
         \item Let $f\in \ccc^\infty(M)$ and $F: \PP_2(M)\to \R$, defined by 
         \begin{align*}
             F(\mu):=\int_M f(x) \mu(\dd x )
         \end{align*} 
         is intrinsically differentiable with 
         \begin{align*}
             D_I^VF(\mu)&= \int_M \skalarq{\nabla f(x),V(x)}_{T_xM}\\
             D_IF(\mu)(x) &= \nabla f(x)
         \end{align*}
         \item Let $\gamma:\R \to M$ be a smooth curve and $f\in \ccc^\infty(M)$ then $F:\PP_2(M) \to M$, defined by 
         \begin{align*}
             F(\mu):= \gamma(\skalarq{f,\mu}) 
         \end{align*}
         is intrinsically differentiable with 
         \begin{align*}
             D^V_IF(\mu)&= \gamma^\prime(\skalarq{f,\mu})\int_M \skalarq{\nabla f(x),V(x)}_{T_xM} \mu(\dd x)\\
              D_IF(\mu)(u) (\cdot) &= \skalarq{\nabla f(u), \cdot}_{T_uM} \gamma^\prime(\skalarq{f,\mu})
         \end{align*}
         \item Let $f\in \ccc^\infty (M\times M)$ then $F: M\times \PP_2(M) \to TM$, defined by
         \begin{align*}
             F(u,\mu)= \int_M \nabla_u f(u,x)\mu(\dd x)
         \end{align*}
         is intrinsically differentiable with 
         \begin{align*}
             D^V_I F(u,\mu)(x)&= \int_M \left(\nabla_x \nabla_u f(u,x)\right) V(x) \mu(\dd x) \\
             D_I F(u,\mu) (x) &= \nabla_x \nabla_u f(u,x)  
         \end{align*}
         
     \end{enumerate}
 \end{example}


In order to do meaningful computations we need good formulae for the intrinsic derivatives, a good idea is the approach due to Proposition 5.35 in \cite{carmona2018probabilistic} Volume-$I$ which shows that by restricting the map to Dirac measures one can show that measure differentiable maps can be identified as smooth maps from spatial variables by the so-called empirical projection map. Since our notion of measure derivative is different from the one in \cite{carmona2018probabilistic}, most importantly since there is no straightforward analogous definition for measures on manifolds, since the space of random variables with values in a compact manifold is not a Hilbert space, therefore a lifting of the map is not useful. For different notions of measure derivatives consider \cite{ren2021derivative} and references therein, they also introduce a type L-derivative which in spirit has similarities with the L-derivative defined in \cite{carmona2018probabilistic}.  The following formula is motivated by Corollary 2.2 \cite{ren2021derivative}, where a related formula has been shown for the L-derivative. 
\begin{lemma}\label{diskreteapprox}
    Let $F:\PP_2(M)\to \R$ be intrinsically differentiable and Lipschitz, consider $F:\otimes_{i=1}^n M \to \R$ defined by 
    \begin{align*}
        (u_1,\dots,u_n) \mapsto F(\frac{1}{n} \sum^n_{i=1} \delta_{u_i})
    \end{align*}
    is $\ccc^1(\otimes^n_{i=1} M)$ for all $u\in M$ with 
    \begin{align*}
        \nabla_{u_i} F(u_1,\dots,u_n)= \frac{1}{n} D_LF(\frac{1}{n}\sum^n_{i=1}\delta_{u_i})(u_i) 
    \end{align*}
 \end{lemma}
\begin{proof}Fix $(u_1,\dots,u_{n-1})\in \otimes^n_{i=1}M$ and $i\in\{1,\dots,n\}$
    We will show that for fixed $(u_1,u_{i-1},u_{i+1},\dots,u_n)$ the map 
    \begin{align*}
        x \mapsto F(\frac{1}{n}\sum^{n-1}_{i=1}\delta_{u_i}+ \frac{1}{n}\delta_{x})
    \end{align*}
    is differentiable. Let first $x\in M\setminus\{u_1,\dots,u_{n-1}\}, v\in T_xM$ and construct a smooth vector field $V$ such that $V(x)= v$ and $V(u_j)=0$ for all $j=1,\dots,n-1$. It follows that $\Phi^V_\eps(u_j)=u_j$ for all $j=1,\dots,n-1$. Then we have 
    \begin{align*}
        &\lim_{\eps \searrow 0} \frac{F(\frac{1}{n} \sum^{n-1}_{i=1} \delta_{u_i}+\frac{1}{n}\delta_{\Phi^V_\eps(x)}-F(\frac{1}{n} \sum^{n-1}_{i=1} \delta_{u_i}+\frac{1}{n}\delta_x)}{\eps}\\
        =&\lim_{\eps \searrow 0} \frac{F(\frac{1}{n} \sum^{n-1}_{i=1} \delta_{\Phi^V_\eps(u_i)}+\frac{1}{n}\delta_{\Phi^V_\eps(x)}-F(\frac{1}{n} \sum^{n-1}_{i=1} \delta_{u_i}+\frac{1}{n}\delta_x)}{\eps} \\
        =& \frac{1}{n} \skalarq{D_LF(\frac{1}{n}\sum^{n-1}_{i=1}\delta_{u_i}+ \frac{1}{n}\delta_{x}), v}_{T_xM}  
    \end{align*}
    This shows that 
    \begin{align*}
          x \mapsto F(\frac{1}{n}\sum^{n-1}_{i=1}\delta_{u_i}+ \frac{1}{n}\delta_{x})
    \end{align*}
    is differentiable for all $x\neq u_i$ for all $i=1,\dots,n-1$. Now let $x=u_i$ for some $i=1,\dots,n-1$. 
    Let $r>0$ be small enough such that $\exp_x(r'v)\notin \{u_1,\dots,u_n\}$ for all $r'<r$, then we get 
    \begin{align*}
    &\frac{F(\frac{1}{n}\sum^{n-1}_{i=1}\delta_{u_i}+ \frac{1}{n}\delta_{\exp_x(rv)} )-F(\frac{1}{n}\sum^{n-1}_{i=1}\delta_{u_i}+ \frac{1}{n}\delta_x )}{r}\\
    &= \frac{1}{r} \int_0^r \frac{d}{d\theta} F(u,\frac{1}{n}\sum^{n-1}_{i=1}\delta_{u_i}+ \frac{1}{n}\delta_{\exp_x(\theta v)}) \dd \theta \\
    &= \frac{1}{r}\int_0^r\frac{1}{n} \skalarq{D_IF(u,\frac{1}{n}\sum^{n-1}_{i=1} \delta_{u_i}+ \frac{1}{n}\delta_{\exp_x(\theta v)})(\exp_x(\theta v)), \Gamma^\theta_0(\exp)(v)}_{T_{\exp_x(\theta v)}M} \dd \theta
    \end{align*}
     Now letting $r\searrow 0$ and using the continuity of $D_I$ we get the result. From the Lipschitz property one can now derive that the limit exists for all curves such that $\gamma^\prime(0)=v$ and $\gamma(0)=x$ since one can easily  prove that $v_\eps=\exp_x^{-1}(\gamma(\eps))= \eps v+ o(\eps)$ and thus 
     \begin{align*}
         \lim_{\eps\searrow 0} \frac{F(\frac{1}{n} \sum^{n-1}_{i=1}\delta_{x_i}+\frac{1}{n}\delta_{\gamma(\eps)})-F(\frac{1}{n}\sum^{n-1}_{i=1}\delta_{u_i}+ \frac{1}{n}\delta_{\exp_x(v_\eps)}) }{\eps}=0       
     \end{align*}
     
 \end{proof}
 One can define higher order intrinsic derivatives for functionals $F:\PP_2(M) \to\R $
 the intrinsic derivative by identifying the intrinsic derivative with the function $D_IF:M\times \PP_2(M)\to TM$. 
 therefore
 \begin{align*}
    D^2_IF(\mu)(x_1,x_2)\in L(T_{x_2}M,T_{x_1}M),D^3_IF(\mu)(x_1,x_2,x_3)\in L(T_{x_3}M,L(T_{x_2}M,T_{x_1}M)),\dots
\end{align*}
We denote the space of smooth functionals and Lipschitz functionals $F:\mathcal{P}_2(M)\to \R$ as $\ccc_1^{\infty,}(\mathcal{P}_2(M))$, these are the functionals such that they are infinitely often intrinsically differentiable and the functions $D^n_IF(\mu)(\cdot,\dots,\cdot)$ are smooth and $D^n_IF(\cdot)(x_1,\dots,x_n)$ are Lipschitz continuous for all $n\in \N,\mu\in \mathcal{P}_2(M),(x_1,\dots,x_n)\in \otimes^{n}_{i=1}M$. One can define similar spaces for functionals $F:\PP_2(M)\to  \R^N$ and $F: M\times \PP_2(M) \to \R^N$ where the spaces will be called $\ccc_1^{\infty}(\PP_2(M),\R^N)$ and $\ccc_1^{\infty,\infty}(M\times \PP_2(M),\R^N)$ respectively. 
Due to technical reasons, in order to compute Krylov-Veretennikov decompositions we will need smoothness assumptions for the driving vector fields of the SDEs \eqref{Wechselwirkung} with respect to the measure variable, since our main technique is embedding the equation into a Euclidean space and enlarge the vector fields smoothly we will need to enlarge the vector fields smoothly with respect to the measure variable. It is not obvious how to do this, therefore we assume our vector fields to be of special form.  
Consider the following classes of vector fields
\begin{align*}
    \mathcal{A}&:= \{V:M\times \PP_2(M) \to TM\vert\\
    & V(u,\mu)= \underset{M\times\dots \times M}{\int\dots\int} \nabla_u \Phi(u,x_1,\dots,x_n)\mu(\dd x_1)\dots\mu(\dd x_n), \;\Phi\in \ccc^\infty(\otimes^{n+1}_{i=1}M)\}\\
    \mathcal{B}&:= \{V:M\times \PP_2(M)\vert V(u,\mu)=g(u,\skalarq{f_1,\mu},\dots,{\skalarq{f_n,\mu}}),\\& \;g:M \times \R^n \to TM,\; g\in \ccc^\infty(M\times \R^n, TM),\; f_1,\dots,f_n\in \ccc^\infty(M) \}
\end{align*}
We will from now on assume that $(V_i)_{i=0,\dots,n} \in \mathcal{A}\cup\mathcal{B}$. The following result allows us to assume that a vector field enlargement can be carried out in a smooth way with respect to space and measure.

\begin{lemma}\label{GlatteErw}
 Let $V\in \mathcal{A}\cup\mathcal{B}$ and let $\iota: M \to \iota(M)\subset \R^N$ be the Nash imbedding, then there exists an extension $\Tilde{V}: \R^N\times \PP_2(\R^N)\to \R^N$ such that $\Tilde{V} \in \ccc^{\infty,\infty,1}(\R^N \times \PP_2(\R^N),\R^N)$ and $\Tilde{V}_{\vert_{\iota(M)\times \PP_2(\iota(M))}}= d \iota(V)$. 
\end{lemma}
\begin{beweis}
    We will only prove the statement in the case when $V(u,\mu)= \int_M \nabla_u \Phi(u,x) \mu(\dd x)$ since for the general case it is similar and for the case $V\in \mathcal{B}$ it is obvious how to proceed. 
First consider the Nash imbedding $\iota:M\to \iota(M) \subset \R^N$ 
then we can extend the function $d\iota_u(\nabla\Phi(u,\iota(\cdot))) $ in the following way, let $(\phi_i,U_i)_{i=1,\dots,m}$ be submanifold charts such that $\phi_i(U_i)$ is a ball. Define for $u=(u_1,\dots,u_N)\in U_j,x=(x_1,\dots,x_N)\in U_k$ 
\begin{align*}
    \Tilde{\nabla_u\Phi}^{j,k}(u,x):= \nabla\Phi(\phi^{-1}_j((u_1,\dots,u_d,0,\dots,0),\phi^{-1}_k(x_1,\dots,x_d,0,\dots,0))
\end{align*}
Now let $(\rho_j)_{j=1,\dots,m+1}$ be a partition of unity subordinate to $(\phi_j(U_j))_{j=1,\dots,m+1}$ where 
\begin{align*}
    \phi_{m+1}(U_{m+1})= \R^N\setminus\overline{(\cup^{m+1}_{j=1}\phi_j(U_j))}
\end{align*}

    then 
    \begin{align*}
       \Tilde{ \nabla_u\Phi(u,x)}= \sum^{m+1}_{j,k=1} \rho_j(u)\rho_k(x) \Tilde{\nabla_u\Phi}^{j,k}(u,x)
    \end{align*}
    is a smooth extension with respect to $(u,x)\in M\times M$, moreover the vector field has compact support and thus all its derivatives are Lipschitz. Hence by an easy computation we can conclude that the extension 
    \begin{align*}
        \Tilde{V}(u,\mu) = \int_M \nabla_u\Tilde{ \Phi}(u,x)\mu(\dd x)  
    \end{align*} 
    has the property $V\in \ccc^{\infty,\infty,1}(\R^N \times \PP_2(\R^N))$, moreover it is easy to see that $\Tilde{V}_{\vert_{\iota(M)\times \PP_2(\iota(M))}}= d\iota(V)$
\end{beweis}

The most important rule for the intrinsic derivative is the chain rule: \begin{lemma}[Chain Rule]\label{Kettenregel}
 Let $F\in \ccc^{1,1}(\PP_2(M))$ be intrinsically differentiable, then we have the following properties.
     Let $x: I \times M \to  M$ where $x_\cdot(u)$ is differentiable and $x: I\times M \to M$ and $\overset{\cdot}{x}:I\times M \to TM$ are continuous then.  
    \begin{align*}
     \frac{d}{d \theta} F(\mu \circ x^{-1}_\theta(\dot))= \int_M \skalarq{D_IF(\mu\circ x^{-1}_\theta(\cdot))(x_\theta(u) ), \overset{\cdot}{x_\theta}(u)}_{T_{x_\theta(u)}M}\mu(\dd u)
    \end{align*}   
 \end{lemma}
     \begin{proof}
Let $\mu^n= \frac{1}{n} \sum^n_{i=1 } \delta_{x_i}$ then we get 
     \begin{align*}
         \frac{d}{d\theta} F(\mu^n\circ x_{\theta}(\cdot)^{-1})&= \sum^n_{i=1} \skalarq{\nabla_{x_i}F(\frac{1}{n}\sum^n_{i=1}\delta_{x_i}),\overset{\cdot}{x}(x_i)}_{T_{x_{x_i}}M}\\
         &= \int_M \skalarq{D_IF(\mu^n\circ x_{\theta}(\cdot)^{-1})(x_\theta(u)),\overset{\cdot}{x}_{\theta}(u)}_{T_{x_\theta(u)}M}  \mu^n(\dd u).
     \end{align*}
Hence 
\begin{align*}
    &\frac{F(\mu^n\circ x_{\theta+\eps}(\cdot)^{-1})-F(\mu^n\circ x_{\theta}(\cdot)^{-1})}{\eps}\\
    =& \frac{1}{\eps}\int_0^\eps \int_M\skalarq{D_IF(\mu^n\circ x_{\theta}(\cdot)^{-1})(x_\theta(u)),\overset{\cdot}{x}_{\theta}(u)}_{T_{x_\theta(u)}}\mu^n(\dd u) \dd \theta
\end{align*}
We want to let $n\to \infty$. If we choose for arbitrary $\mu\in \PP_2(M)$ an appropriate sequence $\mu^n= \frac{1}{n}\sum^n_{i=1} \delta_{x_i}$ such that 
\begin{align*}
    \gamma_2(\mu^n,\mu) \to 0
\end{align*}
 then the theorem is proven by showing, that the right hand side coverges to the correct terms. We will now show that 
\begin{align*}
 &\lim_{n\to \infty}   \int_M\skalarq{D_IF(\mu^n\circ x_{\theta}(\cdot)^{-1})(x_\theta(u)),\overset{\cdot}{x}_{\theta}(u)}_{T_{x_\theta(u)}}\mu^n(\dd u) \\
 =& \int_M\skalarq{D_IF(\mu\circ x_{\theta}(\cdot)^{-1})(x_\theta(u)),\overset{\cdot}{x}_{\theta}(u)}_{T_{x_\theta(u)}}\mu(\dd u)
\end{align*}
 Note that due to continuity $\abs{x_\gamma}_{T_{x_\gamma}M}$ is continuous. 
By continuity of $x_\gamma$, that $\mu_n \circ x^{-1}_\gamma \to \mu \circ x^{-1}_\gamma(\cdot)$ in the Wasserstein topology. Let $\eps>0$ choose $N\in \N$ such that 
\begin{align*}
 \gamma_2(\mu_n\circ x^{-1}_\gamma(\cdot),\mu\circ x^{-1}_\gamma(\cdot))< \eps
\end{align*}
for all $n\ge N$.  Hence for all $n\ge N$ we get
\begin{align*}
      &\abs{\int_M \skalarq{D_If(\mu_n \circ x^{-1}_{\gamma}(\cdot))(x_\gamma(u)),\overset{\cdot}{x}_\gamma}_{T_{x_\gamma}M}\mu_n(\dd u)-  \int_M \skalarq{D_If(\mu \circ x^{-1}_{\gamma}(\cdot))(x_\gamma(u)),\overset{\cdot}{x}_\gamma}_{T_{x_\gamma}M}\mu(\dd u)}\\
      =& \abs{  \int_M \skalarq{D_If(\mu_n \circ x^{-1}_{\gamma}(\cdot))(x_\gamma(u))-D_If(\mu \circ x^{-1}_{\gamma}(\cdot))(x_\gamma(u)),\overset{\cdot}{x}_\gamma}_{T_{x_\gamma}M}\mu_n(\dd u)}\\
      +& \abs{\int_M \skalarq{D_If(\mu \circ x^{-1}_{\gamma}(\cdot))(x_\gamma(u)),\overset{\cdot}{x}_\gamma}_{T_{x_\gamma}M}\mu_n(\dd u)-\int_M \skalarq{D_If(\mu \circ x^{-1}_{\gamma}(\cdot))(x_\gamma(u)),\overset{\cdot}{x}_\gamma}_{T_{x_\gamma}M}\mu(\dd u)}
\end{align*}
The first term can be estimated by the choice of $n\ge N$ and Lipschitz continuity of $D_IF$.

Due to weak convergence one can furthermore choose $n\gg 1$ such that the second term can also be estimated by $\eps$.
We can finally derive the formula.
\begin{align*}
    &\lim_{\eps \searrow 0} \frac{f(\mu\circ x^{-1}_{\theta+\eps})(\cdot)-f(\mu \circ x^{-1}_\theta)(\cdot)}{\eps} \\
    =& \lim_{\eps \searrow 0}\eps^{-1} \int^{\theta+\eps}_\theta \int_M \skalarq{D_If(\mu \circ x^{-1}_{\gamma}(\cdot))(x_\gamma(u)),\overset{\cdot}{x}_\gamma}_{T_{x_\gamma}M}\mu(\dd u)\mu(\dd u) \dd \gamma\\
    =& \int_M \skalarq{D_If(\mu \circ x^{-1}_{\theta}(\cdot))(x_\theta(u)),\overset{\cdot}{x}_\theta}_{T_{x_\theta}M}\mu(\dd u)
\end{align*}
 yielding the result.
It only remains to prove that we can choose good sequence $\mu^n=\frac{1}{n} \sum^n_{i=1} \delta_{u_i}$- Let $X_1.X_2,\dots$ be iid random variables such that $X_i\sim \mu$. Then we have 
\begin{align*}
    \frac{1}{n}\sum^n_{i=1} f(X_i) \to \E(f(X_1))=\int_M f(u) \mu(\dd u) \fastsicher
\end{align*}
for all $f:M\to \R$ such that $\E(\abs{f(X_1)})<\infty$ now take a weak convergence determining family (countable) $(f_k)_{k\ge 0}$. Let $\omega \in \Omega\setminus N$ such that 
\begin{align*}
    \frac{1}{n}\sum^n_{i=1} f_k(X_i(\omega)) \to& \E(f(X_1))=\int_M f(u) \mu(\dd u)\\
    \frac{1}{n}\sum^n_{i=1} d^2_M(x_\gamma(X_i(\omega),x_0) \to& \E(d^2_M(x_\gamma(X_1),x_0))\\
     \frac{1}{n}\sum^n_{i=1} d^2_M(X_i(\omega),x_0) \to& \E(d^2_M(X_1,x_0))\\
\end{align*}
for all $k\ge 0$ and some $x_0\in M$,
which yields the desired result. 

     \end{proof}
\begin{bemerkung}
    Note that the proof can be carried out in exactly the same way for $M=\Rd$, if we assume $\int_{\Rd} \abs{x_\gamma(u)}^2\mu(\dd u)<\infty$ and $\int_{\Rd} \abs{\overset{\cdot}{x_\gamma(u)}}^2 \mu(\dd u)<\infty $.
\end{bemerkung}

     \begin{theorem}
  The solution $x$ to \eqref{Wechselwirkung}  is intrinsically differentiable. For $M=\R^d$ the intrinsic derivative is a solution to the following equation for $\Psi\in \Gamma(\R^N)$:
    \begin{align*}
        d D^\psi_I x_\mu(u,t) &= \nabla F(x_\mu(u,t),\circ \dd t )D^\psi_I x_\mu(u,t) \\
        &+ \int_M D_I F(x_\mu(u,t),\circ \dd t)(x_\mu(r,t))\nabla x_\mu(r,t)\psi(r) \mu(\dd r) \\
        &+ \int_M  D_I F(x_\mu(u,t),\circ \dd t ) (x_\mu(r,t))D^\psi_Ix_\mu(r,t) \mu(\dd r)
    \end{align*}
\end{theorem}
 \begin{proof}
Before we start with the proof note that for $M=\Rd$ Corollary \ref{Lipschitz} also holds for $\nabla x_\mu$. 
     We will prove the statement on $\R^d$ which is enough by the Nash imbedding and Lemma \ref{GlatteErw}.
     Note that the imbedding can be done in a meaningful way since $\iota^n:M^n \to \R^{Nn}
     $ is an isometric embedding whenever $\iota: M \to \R^N $ is a Nash imbedding which is important when $V_i\in \mathcal{A}$.
     The proof is similar as in \cite{wang2021image} but with slight modifications and thus will be carried out here for convenience. Note that the derivative process $\nabla x_\mu$ has the properties of Lemma \ref{Lipschitz}.
We will assume w.l.o.g   that $D_IV_i $ is Lipschitz with respect to to all arguments, this is not a restriction since the extended Vector fields will have bounded support with respect to spatial variables. 
Let $\psi \in \ccc_c^{\infty}(\R^d, \R^d) $ and define $\mu_\eps = \mu\circ \Phi^\psi_\eps(\cdot)^{-1} $, moreover define $\zeta(u,\eps,t)= \frac{x_{\mu_\eps}(u,t)-x_\mu(u,t)}{\eps}$ and $\Tilde{\zeta}(u,\eps,t)=\frac{x_{\mu_\eps}(\Phi^F_\eps(u),t)-x_\mu(u,t)}{\eps}$ our goal is to show 
\begin{align}\label{Kolmogorov}
    \E(\sup_{0\le t \le T}\abs{\zeta(u,\eps,t)-\zeta(v,\delta,t)}^p)\le C(\abs{\eps-\delta}^p +\abs{u-v}^p) 
\end{align}
for all $p\ge 2$. Now define $\eta(u,\eps,r,t)= x_\mu(u,t)+r(x_{\mu_\eps}(\Phi^F_\eps(u),t)-x_\mu(u,t))$, by Lemma \ref{Kettenregel} 
\begin{align}\label{Kette}\begin{split}
    &\frac{d}{dr} V_i(u,\mu\circ (\eta(\cdot,r))^{-1})\\
    =& \int_{\R^d} D_IV_i(u,\mu\circ (\eta(\cdot,r))^{-1})(\eta(v,\eps,r,t))(x_{\mu_\eps}(\Phi_\eps^F(v))-x_\mu(v,t))\mu(\dd v)\\
    =& \eps \int_{\R^d} D_IV_i(u,\mu\circ (\eta(\cdot,r))^{-1})(\eta(v,\eps,r,t))(\zeta(v,\eps,t)-\Tilde{\zeta}(v,\eps,t))\mu(\dd v)
    \end{split} 
\end{align}
by denoting $\Tilde{\eta}(u,\eps,\theta,t)= \theta x_\mu(u,t)+(1-\theta)x_{\mu_\eps}(u,t)$ we get 
\begin{align*}
    &\frac{V_i(x_{\mu_\eps}(u,t),\mu^\eps_t)-V_i(x_\mu(u,t),\mu_t)}{\eps}\\
    =& \frac{1}{\eps} \int_0^1 \frac{d}{d\theta} V_i(\Tilde{\eta}(u,\eps,\theta,t),\mu\circ (\eta(\cdot,\eps,\theta,t))) \dd \theta\\
    =& \int_0^1 \nabla V_i(\Tilde{\eta}(u,\eps,\theta,t),\mu \circ (\eta(\cdot,\eps,\theta,t))) \zeta(u,\eps,t)\\
    +&\int_{\R^d} D_I V_i(\Tilde{\eta}(u,\eps,\theta,t),\mu\circ \eta(\cdot,\eps,\theta,t))(\eta(v,\eps,\theta,t)) \zeta(v,\eps,t) \mu(\dd v)\\
    +& \int_{\R^d} D_I V_i(\Tilde{\eta}(u,\eps,\theta,t),\mu \circ (\eta(\cdot,\eps,\theta,t))^{-1})(\eta(v,\eps,\theta,t))\Tilde{\zeta}(v,\eps,t) \mu(\dd v)\dd \theta.
\end{align*}
In order to show \eqref{Kolmogorov} we need estimates for the Wasserstein distance of our perturbed initial measures. Let w.l.o.g $\eps>\delta>0$ by boundedness of $\psi$ 
\begin{align*}
\E(\gamma_2(\mu_\eps,\mu_\delta)^p)&\le C \gamma_2(\mu_\eps,\mu_\delta)^p\\
&\le \left(\int_{\R^d}  \abs{\int^\eps_\delta \psi(\Phi^\psi_s(u))}^2 \mu(\dd u)\right)^{\frac{p}{2}}\\ &\le C \abs{\eps-\delta}^{p} 
\end{align*}
this estimate works in exactly the same way for $\delta>\eps$. Now note that 
\begin{align*}
    \zeta(u,\eps,t)&= \frac{1}{\eps}\Big( \int_0^t \Tilde{V}_0(x_{\mu_\eps}(u,s),\mu^\eps_s)-\Tilde{V}_0(x_{\mu}(u,s),\mu_s) \dd s\\
  &+  \sum^n_{i=1} \int_0^t V_i(x_{\mu_\eps}(u,s),\mu^\eps_s)-V_i(x_{\mu}(u,s),\mu_s) \dd B^i_s\Big)
\end{align*}
therefore 
\begin{align*}
    \E(\sup_{0\le t\le T} \abs{\zeta(u,\eps,t)}^p)&\le \frac{C}{\eps^p} \E\Big(\int_0^T \abs{x_{\mu_\eps}(u,t)-x_\mu(u,t)}^p + \gamma_2(\mu_\eps,\mu)^p \dd t \Big)\\
    &\le  \frac{C}{\eps^p} \E(\gamma_2(\mu_\eps,\mu)^p)\le C
\end{align*} 
in a similar way one can show 
\begin{align*}
    \E(\sup_{0\le t\le T}\Tilde{\zeta}   (u,\eps,t)^p)\le C.
\end{align*}
Furthermore we have 
\begin{align*}
    &\zeta(u,\eps,T)-\zeta(v,\delta,T)\\
    =& u-v\\
    =&\int_0^1 \int_0^T\nabla \Tilde{V}_0(\Tilde{\eta}(u,\eps,\theta,t),\mu \circ (\eta(\cdot,\eps,\theta,t))) \zeta(u,\eps,t)\\
    -&\Tilde{V}_0(\Tilde{\eta}(v,\delta,\theta,t),\mu \circ (\eta(\cdot,\delta,\theta,t))) \zeta(v,\delta,t) \dd t\\
    +&\int_0^T\int_{\R^d} D_I \Tilde{V}_0(\Tilde{\eta}(u,\eps,\theta,t),\mu\circ \eta(\cdot,\eps,\theta,t))(\eta(q,\eps,\theta,t)) \zeta(q,\eps,t) \mu(\dd q)\\
    -&\int_{\R^d} D_I \Tilde{V}_0(\Tilde{\eta}(v,\delta,\theta,t),\mu\circ \eta(\cdot,\delta,\theta,t))(\eta(q,\delta,\theta,t)) \zeta(q,\delta,t) \mu(\dd q) \dd t\\
    +& \int_0^T\int_{\R^d} D_I \Tilde{V}_0(\Tilde{\eta}(u,\eps,\theta,t),\mu \circ (\eta(\cdot,\eps,\theta,t))^{-1})(\eta(q,\eps,\theta,t))\Tilde{\zeta}(q,\eps,t) \mu(\dd q) \\
    -&\int_{\R^d} D_I \Tilde{V}_0(\Tilde{\eta}(v,\eps,\theta,t),\mu \circ (\eta(\cdot,\eps,\theta,t))^{-1})(\eta(q,\eps,\theta,t))\Tilde{\zeta}(q,\eps,t) \mu(\dd q)     \dd t  \dd \theta\\
    =&  \sum^n_{i=1}\int_0^1\int_0^T \nabla V_i(\Tilde{\eta}(u,\eps,\theta,t),\mu \circ (\eta(\cdot,\eps,\theta,t))) \zeta(u,\eps,t)\\
    -&V_i(\Tilde{\eta}(v,\delta,\theta,t),\mu \circ (\eta(\cdot,\delta,\theta,t))) \zeta(v,\delta,t) \dd B^i_t\\
    +&\int_0^T\int_{\R^d} D_I V_i(\Tilde{\eta}(u,\eps,\theta,t),\mu\circ \eta(\cdot,\eps,\theta,t))(\eta(q,\eps,\theta,t)) \zeta(q,\eps,t) \mu(\dd q)\\
    -&\int_{\R^d} D_I V_i(\Tilde{\eta}(v,\delta,\theta,t),\mu\circ \eta(\cdot,\delta,\theta,t))(\eta(q,\delta,\theta,t)) \zeta(q,\delta,t) \mu(\dd q)\dd B^i_t\\
    +& \int_0^T\int_{\R^d} D_I V_i(\Tilde{\eta}(u,\eps,\theta,t),\mu \circ (\eta(\cdot,\eps,\theta,t))^{-1})(\eta(q,\eps,\theta,t))\Tilde{\zeta}(q,\eps,t) \mu(\dd q)\dd B^i_t\\
    -&\int_{\R^d} D_I V_i(\Tilde{\eta}(v,\eps,\theta,t),\mu \circ (\eta(\cdot,\delta,\theta,t))^{-1})(\eta(q,\delta,\theta,t))\Tilde{\zeta}(q,\delta,t) \mu(\dd q) \dd B^i_t\dd\theta
\end{align*}
In order to estimate these terms by the Lipschitz property of $D_IV_i$ for all $i=1,\dots,n$ and $D_I\Tilde{V}_0$. it is enough to consider the following estimates. 
\begin{align*}
    &\E(\sup_{0\le t \le T}\abs{\Tilde{\eta}(u,\eps,\theta,t)-\Tilde{\eta}(v,\delta,\theta,t)}^p)\\ 
    \le& C\theta^p\E(\sup_{0\le t \le T}\abs{x_\mu(u,t)-x_\mu(v,t)}^p) +C(1-\theta)^p \E(\gamma_2(\mu_\eps,\mu_\delta)^p))\le C(\abs{u-v}^p + \abs{\eps-\delta}^p)
\end{align*}
for all $u,v\in \R^d$ and $\eps,\delta>0$. Furthermore
\begin{align*}
    \int_{\R^d}E(\sup_{0\le t\le T}\abs{\eta(q,\eps,\theta,t)
  -  \eta(q,\eps,\theta,t)}^{  p}) \mu(\dd q)\le C(\theta^p)(\abs{\eps-\delta}^p) 
\end{align*}
moreover 
\begin{align*}
    &\E\left(\sup_{0\le t\le T}\gamma^p_2(\mu\circ \eta(\cdot,\eps,\theta,t)^{-1},\mu\circ\eta(\cdot,\delta,\theta,t)^{-1})\right) \\
    =& \E\left(\left(\sup_{0\le t \le T}\int_{\R^d} \abs{\eta(v,\eps,\theta,t)-\eta(v,\delta,\theta,t)}^{2p}\mu(\dd v)\right)^{\frac{1}{2}}\right)\\
    \le&  \left( \int_{\R^d}\E\left(\left(\sup_{0\le t \le T} \abs{\eta(v,\eps,\theta,t)-\eta(v,\delta,\theta,t)}^{2p}\mu(\dd v)\right)\right)\mu(\dd v)\right)^{\frac{1}{2}}\\
    \le& C \abs{\eps-\delta}^p
\end{align*}
Now by boundedness and Lipschitzness of $V_i$ for $i=1,\dots,n$ and $\Tilde{V}_0$ and all its derivatives and intrisical derivatives we get, 
\begin{align*}
    &\E(\sup_{0\le t \le T}\abs{\zeta(u,\eps,t)-\zeta(v,\delta,t)}^p)\\
    \le& C\int_0^T \E\Big(\abs{\zeta(u,\eps,t)-\zeta(v,\delta,t)}^p \\
    +&\int_{\R^d} \abs{\zeta(u,\eps,t)-\zeta(u,\delta,t)}^p \mu(\dd u )\\
    +&\int_{\R^d} \abs{\Tilde{\zeta}(u,\eps,t)-\Tilde{\zeta}(u,\delta,t)}^p \mu(\dd u )\Big) +C(\abs{\eps-\delta}^p + \abs{u-v}^p)
\end{align*}
from the definition of $\Tilde{\zeta}$ we get 
\begin{align*}
    &\E(\sup_{0\le t\le T}\abs{\Tilde{\zeta}(u,\eps,t)-\Tilde{\zeta}(u,\delta,t)}^p)\\
    =& \int_0^1\E(\sup_{0\le t\le T} \abs{\nabla x_{\mu_\eps}(\Phi^\psi_{r\eps}(u),t)-\nabla x_{\mu_\delta}(\Phi^\psi_{r\delta}(u),t)}^p) \dd r\\
    &\le C \abs{\eps-\delta}^p
\end{align*}
moreover
\begin{align*}
     &\E(\sup_{0\le t \le T}\int_{\R^d}\abs{\zeta(u,\eps,t)-\zeta(u,\delta,t)}^p\mu(\dd u))\\
     \le& C\Big(\int_0^T \E(\int_{\R^d}\abs{\zeta(u,\eps,t)-\zeta(u,\delta,t)}^p\mu(\dd u))\\
     +& \E(\sup_{0\le t\le T}\abs{\Tilde{\zeta}(u,\eps,t)-\Tilde{\zeta}(u,\delta,t)}^p)\\ 
     +& \abs{\eps-\delta}^p
     \Big)
\end{align*}
hence we can finally conclude by means of the Grönwall inequality
\begin{align*}
    \E(\sup_{0\le t \le T}\abs{\zeta(u,\eps,t)-\zeta(v,\delta,t)}^p)\le C(\abs{\eps-\delta}^p+\abs{u-v}^p) 
\end{align*}
by choosing $p\ge 2$ big enough we can conclude with the Kolmogorov continuity theorem that $\zeta$ is continuously extendable to $\eps=0$ yielding the intrinsic differentiability, moreover by
\begin{align*}
    \Tilde{\zeta}(u,\eps,t) \overset{\eps \to 0}{\longrightarrow} \nabla x_\mu(u,t)\psi(u) \fastsicher 
\end{align*}
since 
\begin{align*}
    \Tilde{\zeta}(u,\eps,t)= \int_0^1 \nabla(x_{\mu_\eps}(\Phi_{r\eps}(u))F(\Phi^\psi_{r\eps}(u))) \dd r.  
\end{align*}
In our final step we finalise the proof by observing that 
Let $x$ be the solution to \eqref{Wechselwirkung} and consider the Nash imbedding $\iota: M \to \iota(M)\subset \R^N$ then 
$\iota(x)$ is a solution to the equation in $\R^N$ can be extended to a solution of equations with interactions on $\R^N$ 
\begin{align*}
    \Bar{x}_\mu(u,t)= u+\int_0^t \Bar{\Tilde{V_0}}(\Bar{x}_\mu(u,t),\mu_t)\dd t + \sum^n_{i=1} \Bar{V}_i(x_\mu(u,t),\mu_t) \dd B^i_t
\end{align*}

where $\Bar{V}_i \in \ccc_1^{\infty,1}(\R^N \times \PP_2(\R^N)) $ for all $i=1,\dots,n$ as well as $\Bar{\Tilde{V}}_0 \in \ccc_1^{\infty,1}(\R^N \times \PP_2(\R^N))$. Then the theorem follows since for every $F\in \Gamma(TM)$ there exists a $\Bar{F}\in \ccc_c^{\infty}(\R^N,\R^N)$ such that $\Bar{F}_{\vert_{M}}=F $. And thus projecting the equation on the tangent space of the submanifold yields the result. 
    \end{proof}

We will now prove that the solution to $x$ is Malliavin differentiable.
\begin{lemma}\label{Malliavin}
    Let $(y(u,t))_{u\in M,t\ge 0}$ be Malliavin differentiable such that 
    \begin{align}
        \sup_{u\in M}\sup_{s\in [0,t]}\E(\sup_{t\le T}\abs{D_sy(u,t)}^2)<\infty
    \end{align}
    then we get for $F\in \ccc^{1,1}(\PP_2(M))$ and $\mu_t=\mu\circ y^{-1}(\cdot,t)$ 
    \begin{align}
        D_sF(\mu_t)= \int_M \skalarq{D_IF(\mu_t)(y(u,t)),D_sy(u,t)}\mu(\dd u)
    \end{align}
    
\end{lemma}

    \begin{proof}
   It is enough to prove the result on a submanifold of $\R^N$ by the Nash imbedding. 
        Let $\mu^n= \frac{1}{n}\sum^n_{i=1} \delta_{u_i}$ then by simple chain rule we get 
        \begin{align}\label{formel}
            D_sF(\mu_t^n)= \int_M\skalarq{D_IF(\mu^n_t)(y(u,t)),D_sy(u,t)}_{T_{y(u,t)}M} \mu^n(\dd u)
        \end{align}
        Hence 
        \begin{align*}
            \sup_{n\in \N} \E(\abs{D_sF(\mu^n_t)}^2)<\infty
        \end{align*}
        Since 
        \begin{align*}
            \E(\int_0^T \abs{F(\mu^n_t)-F(\mu_t))}^2 \dd t ) \to  0
        \end{align*}

        it follows from closedness of the Malliavin derivative that $D_sF(\mu_t)$ exists and is the limit of $D_sF(\mu^n_t)$. Now since $F\in \ccc_1^{1}(\PP_2(M))$ one can derive from \eqref{formel} the desired formula, in a similar way as in the proof of the chain rule Lemma \ref{Kettenregel}. Note that a similar idea has already been provided in \cite{djordjevic2021clark} and a similar formula for interacting kernel coefficients has been established.
    \end{proof}

    Now we can prove Malliavin differentiability of \eqref{Wechselwirkung}, the main idea of the proof is motivated from Lemma 3.13 in \cite{nualart2006malliavin}.
    \begin{lemma}

 The solution to \eqref{Wechselwirkung} is Malliavin differentiable and the derivative satisfies 

\begin{align*}
    \sup_{u\in M} \sup_{r\in [0,t]} \E(\sup_{ r\le s \le t} \abs{D_r x(u,s)}^p)<\infty
\end{align*}
for $ p\ge 2$.

\end{lemma}
\begin{beweis}
By the Nash imbedding it suffices to prove the result on $\R^N$.
We prove this by picard iterations, consider for $\mu^0_t \equiv \mu$ then we have for $n\ge 1$
\begin{align*}
    x^n(u,t)=u+   \int_0^t\Tilde{V}_0(x^n(u,s),\mu^{n-1}_s) \dd s +\sum^n_{i=1} \int_0^t V_i(x^n(u,s),\mu_s^{n-1}) \dd B^i_s  
\end{align*}
By induction one can see with, by applying Lemma \ref{Malliavin}, that 

\begin{align*}
    D_sx^n(u,t)  &= \sum^n_{i=1}V_i(x^n(u,s),\mu^{n-1}_s) + \int_s^t \nabla \Tilde{V}_0(x(u,r),\mu_r) D_sx(u,r) \dd r  \\
    &+ \int_s^t \int_{\R^N}D_I \Tilde{V}_0 (\mu^{n-1}_r)(x^{n-1}(u,r)) D_s x^{n-1}(u,r)\mu(\dd u) \dd r \\
    &+ \sum^n_{i=1}\int_s^t\nabla V_i(x^n(u,r),\mu^{n-1}_r) D_s x^n(u,r) \dd B^i_r \\
    &+\sum^n_{i=1}\int_s^t \int_{\R^N} D_IV^i(\mu_t^{n-1}) (x^{n-1}(u,r)) D_s x^{n-1}(u,r)\mu_0(\dd u) \dd B^i_r
\end{align*}
Therefore

\begin{align*}
    & \sup_{0\le r \le t}\E(\sup_{r\le s \le T} \abs{D_r x^{n}(u,s)}^2) \\
    &\le C(T,r)(1+ \int_r^T \E((D_r x^{n}(u,\rho))^2) \dd \rho\\
    &+ \int_s^T  \sup_{u\in \Rd}\E(\abs{D_rx^{n-1}(u,t)}^2)   
\end{align*}
by Grönwall's Lemma we get 
\begin{align*}
 \sup_{u\in \Rd} \E(\sup_{r\le s \le T} \abs{D_r x^{n}(u,s))}^2) \le C(1+  \int_s^T  \E(\abs{D_rx^{n-1}(u,s)}^2) \dd s  ) 
\end{align*}
Which guarantees that the $\lp{2}$ norm of the derivatives is bounded in $n$ since 
\begin{align*}
    \E(\sup_{0\le t \le T} \abs{x^n(u,t)-x(u,t)}^2) \to 0
\end{align*}
Lemma 1.5.4 in \cite{nualart2006malliavin} now guarantees that, $x$ is Malliavin differentiable. The result therefore follows by differentiating the solution.
\end{beweis}
\section{Krylov-Veretennikov decomposition}

We first need an Itô formula for the measure valued process befor we can finalise the proof of the Krylov-Veretennikov decomposition. 
\begin{lemma}\label{MassIto}
    Let $F\in \ccc_1^{\infty}(\PP_2(M))$ then 
    \begin{align*}
        F(\mu_t)= F(\mu_0) + \int_0^t D_I^{\Tilde{V}_0(\cdot,\mu_s)}F(\mu_s)\dd s + \sum^n_{i=1} \int_0^t D^{V_i(\cdot,\mu_s)}F(\mu_s)\dd B^{i}_s
    \end{align*}
\end{lemma}
\begin{proof}
Note first that by Lemma \ref{diskreteapprox} we get
\begin{align*}
    D^{V}_IF(\mu^n)= \sum^n_{i=1}V\Tilde{F}(u_1,\cdot,u_n)(u_i)
\end{align*}

    Let us consider first $\mu^n_t=\frac{1}{n}\sum^n_{j=1} \delta_{x_{\mu^n}(u_j,t)}$  Then by the usual Itô formula we get 
    \begin{align*}
       F(\mu^n_t)= F(\mu^n) + \int_0^t D^{\Tilde{V}_0(\cdot,\mu^n_s)}_I F(\mu^n_s) \dd s +\sum^n_{i=1} \int_0^t D^{V_i(\cdot,\mu^n_s)}F(\mu^n_s) \dd B^i_s
    \end{align*}
    hence we can w.l.o.g approximate general $(\mu_t)_{t \ge 0}$ by Lemma \ref{Massabsch} and compute the limit by continuity of all terms involved we get the result.
\end{proof}
The proof of the following result is almost entirely based on the proof of Theorem 1.14 in \cite{dorogovtsev2012krylov} but due to discrepancies in proof and statement carried out here again for then convenience of the reader.
 \begin{theorem}\label{Krylov-Veretennikov}
        Let $f\in \ccc_1^{\infty}(\PP_2(M))$ if $T_tf \in \ccc^{\infty}(\PP_2(M))$ and  if the coefficients of the Itô-Wiener decomoposition are continuous with respect to time then 
        \begin{align*}
            f(\mu_t) =    T_tf(\mu) +\sum^n_{i=1}\sum^\infty_{k=1} \underset{\Delta^k([0,t])}{\int \dots \int} T_{\tau_1}A_i T_{\tau_2-\tau_1}\dots T_{\tau_k-\tau_{k-1}}A_iT_{t-\tau_k}f(\mu) \dd B^i(\tau_1)\dots \dd B^i(\tau_k)    \end{align*}
            with $A_i = D_I^{V_i(\cdot,\mu)}$ for $i=1,\dots,n$.
    \end{theorem}     
    \begin{proof} Note first that 
    \begin{align*}
        f(\mu_t)= \E(f(\mu_t)) +\sum^\infty_{k=1}\sum^n_{i=1}\underset{\Delta^k([0,t])}{\int }a^{t,i}_k(\mu,f,\tau_1,\dots,\tau_k) \dd B^i(\tau_1)\dots \dd B^i(\tau_k)
    \end{align*}
        Let $t,s\ge 0$ then due to the fact that 
        \begin{align*}
            f(\mu_{t+s})= f(\Theta_t(\mu_0^s)(\mu_t))
        \end{align*}
we need to compute
        \begin{align*}
        a_0(\mu_t,f)&= T_t a_0(\cdot,f)(\mu)+  \sum^n_{i=1}\int_0^t a^{t,i}_1(\mu,a^s_0(\cdot,f),s) \dd B^i_s\\
        &+ \sum^\infty_{k=2} \underset{\Delta^k([0,t])}{\int \dots\int}a^{t,i}_k(\mu,a^s_0(\cdot,f),\tau_1,\dots,\tau_k) \dd B^i_{\tau_1}\dots\dd B^i_{\tau_k}
        \end{align*} 
        and 
        \begin{align*}
           a^s_1(\mu_t,f,\tau_1-t) = T_ta_1^s(\cdot,f,\tau_1-t)+ \sum^\infty_{k=1} a^t_k(\mu,a_1^s(\cdot,f,\tau_1-t)) 
        \end{align*}
        therefore 
        \begin{align*}
            a^{t+s}_1(\mu,f,\tau_1) = a^t_1(\mu,a^s_0(\cdot,f),\tau_1)1_{0\le\tau_1\le t}+T_ta^s_1(\cdot,f,\tau_1-t)1_{t\le \tau_1\le t+s}
        \end{align*}
        therefore we can get 
        \begin{align*}
            a^{t,i}_1(\mu,f,0)&= a^{t-\tau_1,i}_1(\mu,a_0^{\tau_1,i}(\cdot,f),0)\\
            a^{t,i}_1(\mu,f,\tau_1)&= T_{\tau_1}a^{t-\tau_1,i}_1(\mu,f,0)   
        \end{align*}
Now let $T_{\tau_1}f(\mu)$ and then we get by 
    \begin{align*}
      AT_{t-\tau_1}f(\mu)&= \lim_{T\searrow 0} \frac{1}{T} \sum^n_{i=1}\E(T_{{t-\tau_1}}f(\mu_T)B^i_T)\\
      &= \lim_{T\searrow 0} \frac{1}{T} \E\left(\int_0^T a^{T,i}_1(\mu,T_{t-\tau_1}f(\cdot),s) \dd s\right)= a^0_1(\mu,T_{t-\tau_1}f(\cdot),0) 
    \end{align*}
    Hence 
    \begin{align*}
        a^{t,i}_1(u,\tau_1)=T_{\tau_1} A_iT_{t-\tau_1} f(\mu) 
    \end{align*}
 The rest can be proven by induction, assume that the statement holds for all $j\le k$
then we get:
 \begin{align*}
     a^{t+s,i}_{k+1}(\mu,f,\tau_1,\dots,\tau_{k+1})1_{0\le \tau_1\le t \le \tau_2 \dots  \le \tau_{k+1}\le t+s} &= a^{t,i}_{1}(\mu,a^s_k(\cdot,f,\tau_{2}-t,\dots,\tau_{k+1}-t),\tau_{1})\\
     &= T_{\tau_{1}}A_i T_{\tau_{1}-t}a^{s,i}_{k}(\mu,f,\tau_{1}-t,\dots,\tau_k-t)\\
     &= T_{\tau_{1}}A_i T_{t-\tau_{1}}T_{\tau_{2}-t}AT_{\tau_3-\tau_2}\dots AT_{t+s-\tau_{k+1}}\\
     &= T_{\tau_1}A_i T_{\tau_2-\tau_1}A_i \dots T_{\tau_{k+1}-\tau_k}A_iT_{t+s-\tau_{k+1}}
 \end{align*}
 for all $t,s\ge 0$ proving the end result.
    \end{proof}

Now in order to suffice the assumptions of the Theorem \ref{Krylov-Veretennikov} we have yet to prove the smoothness with respect to the measure and with respect to the noise. 

The remainder is proven for $M=\Rd$ since by embedding the equation into a euclidean space and enlarging the vector fields in the same way as in Lemma \ref{GlatteErw},   that case is enough to treat. 

\begin{theorem}
 The solution to the SDE \eqref{Wechselwirkung} is infinitely often differentiable with respect to the measure variable. 
\end{theorem}
\begin{proof}
First of all note that 
  \begin{align*}
        d D^\Phi_I x_\mu(u,t) &= \nabla F(x_\mu(u,t), \dd t )D^\Phi_I x_\mu(u,t) \\
        &+ \int_{\Rd} D_I F(x_\mu(u,t), \dd t)(x_\mu(v,t))D x_\mu(v,t)\Phi(v) \mu(\dd v) \\
        &+ \int_{\Rd}  D_I F(x_\mu(u,t), \dd t ) (x_\mu(v,t))D^\Phi_Ix_\mu(v,t) \mu(\dd v)
    \end{align*} 
    therefore 
    \begin{align*}
        D_I x_\mu(u,t)(v)&= \nabla F(x_\mu(u,t),\dd t) D_Ix_\mu(u,t)(v) \\
        &+ D_I F(x_\mu(u,t), \dd t)(x_\mu(v,t))D x_\mu(v,t) \\
        &+ \int_{\Rd}D_I F(x_\mu(u,t),\dd t )(x_\mu(r,t))D_Ix_\mu(r,t)(v)\mu(\dd r)
    \end{align*}
    Now since 
    \begin{align*}
       d \nabla x_\mu(u,t)= \nabla F(x_\mu(u,t),\circ \dd t)\nabla x_\mu(u,t)
    \end{align*}
Rewriting the equation into Itô form yields:
    \begin{align*}
        d\nabla x_\mu(u,t) =\Phi(x_\mu(u,t),\mu_t)\nabla x_\mu(u,t)\dd t +\sum^n_{i=1} \Psi^i(x_\mu(u,t),\mu_t)\nabla x_\mu(u,t)\dd B^i _t
    \end{align*}
    therefore we have a linear equation with coefficients that are intrinsically differentiable hence it is also intrinsically differentiable. A similar argument can be carried out for $D_I x_\mu(u,t)(v)$ and higher order derivatives.
\end{proof}

\begin{lemma}
Let $f\in \ccc_1^{1}(\PP_2(M))$ then the map $\mu\to \E(f(\mu_t))$ is intrinsically differentiable and 

\begin{align*}
    D^V_I \E(f(\mu_t)) &= \E\Big(\int_M \skalarq{D_If(\mu_t)(x_\mu(u,t)), (dx_\mu(\cdot,t))_u(V)}_{T_{x_\mu(u,t)}} \mu(\dd u) \\
    &+\int_M \skalarq{D_If (\mu_t)(x_\mu(u,t)),D^V_Ix_\mu(u,t)}T_{x_\mu(u,t)} \mu(\dd u)\Big)
\end{align*}
\end{lemma}
\begin{beweis} Let $V\in\Gamma(TM)$ and $\mu\in \PP_2(M)$ denote $\mu_\gamma= \mu\circ \left(\Phi^V_\gamma(\cdot)\right)^{-1}$ for $\gamma\ge 0$. 
    Furthermore consider the map $F(\eps,\delta)= E(f(\mu^\eps_t \circ x^{-1}_{\mu^\delta}(\cdot,t)))$, now by Lemma \ref{Kettenregel} we get 
    \begin{align*}
        \frac{d}{d\eps}_{\vert_\eps=0} F(\eps,\eps)&= \frac{\partial}{\partial \eps}F(0,0) + \frac{\partial}{\partial \delta} F(0,0)  \\
        &=  \E\Big(\int_M \skalarq{D_If(\mu_t)(x_\mu(u,t)), (dx_\mu(\cdot,t))_u(V)}_{T_{x_\mu(u,t)}} \mu(\dd u) \\
    &+\int_M \skalarq{D_If (\mu_t)(x_\mu(u,t)),D^V_Ix_\mu(u,t)}T_{x_\mu(u,t)} \mu(\dd u)\Big)
    \end{align*}
\end{beweis}
\begin{bemerkung}
     It follows that whenever $f\in \ccc_1^{\infty}(\PP_2(M))$ then $P_tf\in \ccc^\infty(\PP_2(M))$.
\end{bemerkung}

\begin{theorem}
    The solution is infinitely often Malliavin differentiable and the Malliavin derivative is continous. 
\end{theorem}
\begin{proof}
    First note that the solution is Malliavin differentiable 
      \begin{align*}
        D_s x_\mu(u,t)&=\sum^n_{i=1}V_i(x_\mu(u,s),\mu_s)\\
        &+ \int_s^t \nabla F(x_\mu(u,t),\circ \dd t)D_s x(u,r) \\
        &+ \int_M \int_s^t D_IF(x_\mu(u,t),\circ \dd r)D_s x(u,r) \mu(\dd u)  
    \end{align*}
    $D_sx_\mu(u,t)$ solves a SDE with good coefficients as well such that it is also Malliavin differentiable. 
\end{proof}
We thus know that all the Malliavin derivatives are continuous and therefore from the formula 
\begin{align*}
    D_sF(\mu_t)= \int_M \skalarq{D_IF(\mu_t)(x_\mu(u,t)),D_sx_\mu(u,t)}_{T_{x_\mu(u,t)}}\mu(\dd u )   
\end{align*}
is continuous with respect to $t$ and $s$. The same holds for higher order Malliavin derivatives. Hence we can conclude the Krylov-Veretennikov decomposition. 
\section*{Acknowledgements}
The second author wants to thank Max von Renesse, Jonas Hirsch and Dominik Inauen for helpful conversations about differential geometry, isometric embeddings and stochastic analysis on manifolds.
\printbibliography
\end{document}